\documentclass{article}

 \usepackage{geometry}
\usepackage{amsmath,amsfonts, amsthm,amssymb,oldgerm}
\usepackage{graphicx}
\usepackage{array}
\usepackage{epsfig}
\usepackage{relsize}
\def\r{\mathbb{R}}
\def\H{\mathcal{H}}

\def\y{\mathcal{Y}}
\def\C{\mathbb{C}}

\def\<{\langle}
\def\>{\rangle}
\def\D{\boldsymbol{\mathcal{D}}}

\def\2{L^2}
\def\c{L^2([0,T];H^1(0,1))}
\def\i{\infty}
\def\h{\hat}
\def\v{\vec}
\def\lam{\lambda}
\def\ds{\displaystyle}
\newtheorem{thm}{\bf Theorem}[section]
\newtheorem{defn}[thm]{\bf Definition}
\newtheorem{lem}[thm]{\bf Lemma}

\newtheorem{prop}[thm]{\bf{Proposition}}

\newcommand{\rinfty}{\int ^{\infty}_{-\infty}}
\title{Well-posedness of  a Class of  Non-homogeneous Boundary Value Problems of
 the Korteweg-de Vries  Equation on a Finite Domain}
\author{
 Eugene Kramer
\\ \small{Department of Mathematics, Physics,
and Computer Science} \\ \small{University of Cincinnati, Raymond
Walters College}\\ \small{Cincinnati, Ohio 45236} \\
\small{eugene.f.kramer@uc.edu }\\ \qquad \vspace{.2in}
\\Ivonne Rivas\\
{\small Department of Mathematical Sciences,} \  \\
{\small University of Cincinnati } \\ {\small Cincinnati, Oh 45221} \\
{\small
email: rivasie@mail.uc.edu} \\
\quad \vspace{.2in}\\ Bing-Yu Zhang \\ {\small Department of
Mathematical Sciences} \\ {\small University of Cincinnati}
\\ {\small Cincinnati,
Ohio 45221} \\ {\small email: zhangb@ucmail.uc.edu}}

\date{}

\begin{document}

\maketitle
\newpage
\begin{abstract}
In this paper, we  study  a class of  initial and boundary value
problems   proposed by Colin and Ghidalia for the Korteweg-de Vries
equation posed on a bounded domain $(0,L)$.  We show that the
initial-value problem is locally well-posed in the classical Sobolev
space  $H^s(0,L)$ for $s>-\frac34$, which provides a positive answer
to one of the open questions of Colin and Ghidalia \cite{ColGhi01}.
\end{abstract}

\section{Introduction}
In  this paper we study a class of  initial-boundary value problem
(IBVP)  for  the Korteweg-de Vries (KdV) equation posed on a finite
domain with nonhomogeneous boundary conditions,
    \begin{equation}\label{PP}
        \begin{cases}
        u_t+u_x+u_{xxx}+uu_x=0,\qquad u(x,0)=\phi(x),  & x\in (0,L),  \ t\in
        \r^+ ,
        \\
        u(0,t)=h_1(t),\ u_x(L,x)=h_2(t),\ u_{xx}(L,t)=h_3(t).
        \end{cases}
    \end{equation}
This IBVP can be considered as a  model for propagation of surface
water waves in the situation where a wave-maker is putting energy in
a finite-length channel from the left $(x=0)$ while the right end
$(x=L)$ of the channel is free (corresponding the case of
$h_2=h_3=0$) (see \cite{c-g-1}). The problem was first proposed and
studied by Colin and Ghidaglia in the late 1990s
 \cite{c-g-1,c-g-2, ColGhi01}. In particular, they investigated the well-posedness of the IBVP in the
classical Sobolev space $H^s (0,L)$ and  obtained the following
results.

\medskip
\noindent {\bf Theorem A} (\cite{ColGhi01})

\begin{itemize} \item[(i)]  \emph{Given
$h_j\in C^1([0, \infty)), \ j=1,2,3$ and $\phi \in H^1 (0,L)$
satisfying $h_1(0)=\phi (0)$, there exists a $T>0$ such that the
IBVP (\ref{PP}) admits a solution (in the sense of distribution)}
\[ u\in L^{\infty}(0,T; H^1(0,L))\cap C([0,T]; L^2 (0,L)) .\]

\item[(ii)] \emph{The solution $u$ of the IBVP
(\ref{PP}) exists globally in $H^1(0,L)$ if the size of its initial
value $\phi \in H^1 (0,L)$ and its boundary values $h_j\in C^1([0,
\infty )), \ j=1,2,3$ are all small.}
\end{itemize}
In addition, they showed that the associate linear IBVP
\begin{equation}\label{LPP}
        \begin{cases}
        u_t+u_x+u_{xxx}=0,\qquad u(x,0)=\phi(x)  & x\in (0,L),  \ t\in \r^+ \\
        u(0,t)=0,\ u_x(L,x)=0,\ u_{xx}(L,t)=0
        \end{cases}
    \end{equation}
 possesses the following  smoothing property:

 \smallskip
\emph{ For any $\phi \in L^2 (0,L)$, the linear IBVP (\ref{LPP})
admits a
 unique solution $$u\in C(\r^+; L^2 (0,L))\cap L^2 _{loc} (\r^+; H^1
 (0,L)).$$}

 \smallskip
 \noindent
 Aided by this smoothing property, Colin and Ghidaglia  showed
 that the homogeneous IBVP (\ref{PP}) is locally well-posed in the
 space $L^2 (0,L)$.

 \medskip
 \noindent
 {\bf Theorem B} (\cite{ColGhi01})

 \quad

 \emph{Assuming $h_1=h_2=h_3\equiv 0$, then for any $\phi \in L^2
(0,L)$, there exists a $T>0$ such that the IBVP (\ref{PP}) admits a
unique weak solution $u\in C([0,T]; L^2 (0,L))\cap L^2 (0,T; H^1
(0,L))$.}

\medskip

The well-posedness results presented in Theorem A are not in the
full strength of the well-posedness in the sense of Hadamard since
both uniqueness and continuous dependence are missing, in
particular,  for  the IBVP (\ref{PP}) with   nonhomogeneous boundary
conditions. To encourage further investigation, a series of open
problems were proposed  by Colin and Ghidaglia in \cite{ColGhi01},
of which, two of them are listed below.

\bigskip
\noindent {\bf Problems}

\begin{itemize}
\item[(1)] \emph{Is it possible to prove  global existence of solutions of
(\ref{PP}) for e.g. smooth solutions (as it in the case for both
quarter plane and the whole line cases)?}

It is remarked by Colin and Ghidaglia in \cite{ColGhi01}: ``
\emph{for these problems, uniqueness rely on a priori estimate in
$H^2$ that we are not able to extend here and therefore establish
the existence of more regular solutions.}"

\item[(2)] \emph{Is it possible to establish the existence of solutions
of (\ref{PP}) with their initial value in
the space $H^s (0,L) $ for some $s<0$ as in the case of the whole
line?}

\smallskip
 Colin and Ghidaglia expected the answer to be positive because of the
 the strong  smoothing property of the associated linear IBVP
 (\ref{LPP}).

\end{itemize}
\newcommand{\norm}[1]{\| #1\|}

In this paper, we will  continue Colin and Ghidalia's work
\cite{c-g-1, c-g-2, ColGhi01} to study the well-posedness problem of
the IBVP (\ref{PP}) in the space $H^s (0,L)$.   We aim at 1)
establishing the well-posedness of the IBVP (\ref{PP}) in the full
strength of Hadamard including \emph{existence, uniqueness and
continuous dependence} and 2) showing that the IBVP (\ref{PP}) is
(locally) well-posed in the space $H^s (0,L)$ when $s\geq 0$  and
$-\frac34 < s< 0$.

  \medskip
  In order to describe precisely our results, we
introduce the some notations.

\smallskip
 For given $T>0$ and $s\in \mathbb{R}$, let
\[ \mathbb{H}^s(0,T) := H^{\frac{s+1}{3} }(0,T) \times
H^{\frac{s}{3}}(0,T)\times H^{\frac{s-1}{3}} (0,T), \]
\[ D_{s,T} := H^s(0,L) \times \mathbb{H}^s (0,T) \]
and \[ Z_{s,T}= C([0,T]; H^s (0,L))\cap L^2 (0,T; H^{s+1}(0,L))\]
 For
the well-posedness of the IBVP (\ref{PP}), we intend to establish in
this paper,  some compatibility conditions relating the initial
datum $\phi(x) $ and the boundary data $h_j(t), j=1,2,3$ are needed.
A simple computation shows that if $u $ is a $C^{\infty}$-smooth
solution of the IBVP (\ref{PP}), then its initial data $u(x,0) =
\phi (x)$ and its boundary values $h_j(t), \ j=1,2,3$ must satisfy
the following compatibility conditions:
\begin{equation}
 \phi _k (0 ) = h^{(k)}_1 (0), \quad \phi _k' (L  ) = h^{(k)}_2 (0),
 \quad
\phi _k'' (L ) = h^{(k)}_3  (0) \label{y5-1}
 \end{equation}
for $k=0 \ , 1, \, \cdots , $ where $h_j^{(k)} (t) $ is the $k-$th
order derivative of $h_j $ and
\begin{equation}
\left \{ \begin{array}{l}
\phi _0 (x) = \phi (x) \\ \\
\phi _k (x) = -\left (\phi _{k-1}^{'''} (x) + \phi _{k-1}^{'}(x)  +
 \sum _{j=0}^{k-1} \left (\phi _{j}(x)
\phi _{k-j-1} (x) \right )' \right ) \end{array} \right.
\label{y6-1}
\end{equation}
for $k=1, \, 2, \cdots .$   When the well-posedness of (\ref{PP}) is
considered in the space $H^s (0,L)$ for some $s\geq 0$, the
following $s-$compatibility conditions thus arise naturally.
\begin{defn}($s-$compatibility )  Let $T>0 $
and $s\geq 0$ be given. A four-tuple $$(\phi , \vec{h})= (\phi , h_1
, \, h_2 ,\, h_3  )\in  D_{s,T} $$ is said to be $s-$compatible with
respect to the IBVP (\ref{PP}) if
 \begin{equation}\label{comp1}\phi_k(0)=h^{(k)}_1(0)
\end{equation}
when $k=0,1,\ldots [s/3]-1$ and $s-3[s/3]\le 1/2,$
\begin{equation}\label{comp2}\phi_k(0)=h^{(k)}_1(0), \quad \phi_k '(1)=h_2^{(k)}(0)\end{equation}
when $k=0,1,\ldots [s/3]$ and $1/2 < s-3[s/3]\le 3/2$ and
  \begin{equation}\label{comp3}
\phi_k(0)=h_1^{(k)}(0), \quad \phi_k '(1)=h_2^{(k)}(0) \quad \phi_k
''(1)=h_3^{(k)}(0) \end{equation} when $k=0,1,\ldots [s/3]+1$ and
$3/2 \le s-3[s/3]\le 9/2.$
We  adopt the convention that (\ref{comp1}) is vacuous if $[s/3]-1< 0$.\\
\end{defn}

As one of the main results in this paper, the following theorem
states that the IBVP (\ref{PP}) is locally well-posed in the space
$H^s (0,L)$ for any $s\geq0$.
\begin{thm}  Let $s\geq 0 $,  $T>0$ and $r>0$  be given with
\[ s\ne \frac{2j-1}{2}, \quad j=1,2,\cdots .\]
 There exists
 a $T^*\in (0,T]$ such that for any  $s-$compatible
 \[ (\phi, \vec{h} )\in D_{s,T} \] satisfying \[ \| (\phi ,
\vec{h})\|_{D_{s,T}}  \leq r, \] the IBVP (\ref{PP}) admits a unique
solution \[ u\in   C([0,T^*]; H^s (0,L))\cap L^2 (0,T^*; H^{s+1}
(0,L)).\] Moreover, the corresponding solution map is Liptschitz
continuous \footnote{The solution map, is in fact, real analytic
(cf. \cite{zh2,zh3,Zhang92}}.
\end{thm}
To get the well-posedness of the IBVP (\ref{PP}) in the space $H^s
(0,L)$ with $s<0$,  the following Bourgain spaces are needed (cf.
\cite{ColKen02, Holmer06, BSZ06}).

 For any given $s\in \mathbb{R}$,
$0\leq b\leq 1 $, $0\leq \alpha \leq 1$ and function $w\equiv
w(x,t): \, \mathbb{R}^2 \to \mathbb{R}$, define
\begin{eqnarray}
& & \Lambda _{s,b} (w) =\left (
\rinfty\rinfty<\tau-(\xi^3+\xi)>^{2b} <\xi >^{2s}
\left |\hat{w}(\xi, \tau )\right |^2 d\xi d\tau \right )^{\frac12} \, ,\nonumber \\
&& \lambda _{\alpha } (w) =  \left ( \rinfty\int _{|\xi |\leq 1}
<\tau >^{2\alpha } \left |\hat{w}(\xi, \tau )\right |^2 d\xi d\tau
\right )^{\frac12}  \label{3.11}
\end{eqnarray}
where $<\cdot>:= (1+|\cdot|^2)^{\frac12}$.  In addition, define
\[ {\cal G}_{s} (w) = \left ( \rinfty (1+|\xi |)^{2s} \left ( \rinfty \frac{|\hat{w} (\xi ,
 \tau )|}{1+ |\tau - (\xi ^3-\xi )|} d \tau \right )^2 d\xi \right )^{1/2}, \]

\[ {\cal Q}_{s,b} (w) = \left ( \rinfty \rinfty (1+|\xi |)^{2s}   \frac{|\hat{w} (\xi , \tau )
|^2}{\Big (1+ |\tau - (\xi ^3-\xi )|\Big )^{2b}} d \tau   d\xi \right )^{1/2} \]
and
\[ {\cal P}_{\alpha} (w) = \left ( \rinfty   \int _{|\xi |\leq 1} \frac{|\hat{w} (\xi ,
 \tau )|^2}{(1+ |\tau|  )^{2(1-\alpha )}} d \tau   d\xi \right )^{1/2} .\]
Let $X_{s,b}  $ be the space of  all functions $w$ satisfying
\[ \norm{w} _{X_{s,b}}  :=   \Lambda _{s,b}  (w)  < \infty \]
  while  $Y_{s,b} $  is the space
of  all $w$ satisfying
\[ \norm{w}_{Y _{s,b}} : = \left (  {\cal G}^2_{s}(w) +{\cal Q}^2_{s,b} (w) \right )^{1/2} < \infty .\]
In addition, let $X_{s,b}^{\alpha } $ be the space of  all functions
$w$ satisfying
\[ \norm{w}_{X^{\alpha }_{s,b}}  :=  \left (\Lambda _{s,b}^2 (w) + \lambda _{\alpha }^2 (w)
\right ) ^{1/2} < \infty \]   and let $Y_{s,b}^{\alpha }$ be the
space of  all $w$ satisfying
\[ \norm{w}_{Y^{\alpha }_{s,b}} : = \left ({\cal P} _{\alpha }^2 (w) + {\cal G}^2_{s}(w) +{\cal Q}^2_{s,b} (w) \right )^{1/2} < \infty .\]
The spaces $X_{s,b}$, $Y_{s,b}$, $X^{\alpha }_{s,b} $ and $Y^{\alpha
}_{s,b}$ are all Banach spaces.  Note that  $X_{s,b}$ and $X^{\alpha
}_{s,b} $ are equivalent when $b\geq \alpha $. The spaces $Y_{s,b}$
and $X_{s, -b}$ are also equivalent when $b< \frac12 $. Define also
\[ {\cal X}^{\alpha }_{s,b} \equiv C(R;H^s (R)) \cap X_{s,b}^{\alpha } \]
with the norm
\[ \norm{w}_{{\cal X}^{\alpha }_{s,b}} = \left ( \sup _{t\in R}
\norm{w(\cdot, t)}_{H^s (R)}^2 + \norm{w}_{X_{s,b}^{\alpha }}^2
\right )^{1/2} .\]

The above Bourgain-type spaces are defined for functions posed on
the whole plane $\mathbb{R}\times  \r$. However,  the IBVP
(\ref{PP}) is posed  on the finite domain $(0,L)\times (0,T)$.  It
is thus natural to   define a restricted  version of the Bourgain
space $X_{s,b}$ to the  domain $(0,L)\times (0,T)$  as follows:
\[ X_{s,b}^T  = X_{s,b}\Big |_{(0,L)\times
(0,T) }\] with the quotient norm
\[ \norm{u}_{X_{s,b}^T}  \equiv \inf _{w\in X_{s,b}}\{\norm{w}_{X_{s,b}} : \ w(x,t) =u(x,t)
\ \mbox{on} \ (0,L)\times (0,T) \}\] for any given function $u(x,
t)$ defined on $(0,L)\times (0,T)$.   The spaces $Y_{s,b}^T  $,
$X^{\alpha ,T}_{s,b}  $, $Y^{\alpha ,T}_{s,b} $ and $ {\cal
X}^{\alpha ,T }_{s,b} $ are defined similarly.

\medskip
Next theorem, another main result of this paper, provides a positive
answer to Problem (2) listed earlier.

\begin{thm} Let $s\in (-\frac34,0) $, $T>0$ and $r>0$ be given. There exist
  $T^*\in (0,T]$, $\alpha >\frac12$  and $0< b< \frac12$ such that for any
 \[ (\phi, \vec{h} )\in D_{s,T}\] satisfying \[ \| (\phi ,
\vec{h})\|_{D_{s,T}} \leq r, \]
the IBVP (\ref{PP}) admits a unique
solution
\[ u\in C([0,T^*]; H^s (0,L))\cap X^{\alpha, T^* }_{s,b}.\]
Moreover, the corresponding solution map is Liptschitz  continuous.
\end{thm}

The following remarks are in order.

\begin{itemize}
\item[(i)] According to Theorem 1.2, the IBVP (\ref{PP}) is
well-posed in the space $H^s (0, L)$ for any $s\geq 0$, not just for
$s=0$ or $s=1$. In particular, it demonstrates the existence of
classical solutions and shows that the smoother of the initial value
and boundary data, the smoother the  corresponding solution.
\item[(ii)] In order to have solution $u$ in the space $C(0,T]; H^s
(0,L))$, Theorem 1.2 only requires that its initial value $\phi \in
H^s (0,L)$ and its boundary data
\begin{equation}\label{cond-1} h_1\in H^{\frac{s+1}{3}} (0,T),\quad h_2 \in
H^{\frac{s}{3}} (0,T), \quad h_3 \in H^{\frac{s-1}{3}}
(0,T).\end{equation}
 In particular, if $s=1$, it is sufficient to
require that
\[ h_1\in H^{\frac13} (0,T), \quad h_2 \in L^2 (0,T), \quad h_3 \in
H^{-\frac13} (0,T) ,
\] rather than $h_j\in C^1(0,T), j=1,2,3$ as in Theorem A.   Moreover, the
condition (\ref{cond-1}) is optimal in order to have the
corresponding solution $u\in C([0,T];H^s (0,L))$.
\item[(iii)] Taking hint from   the recent works of Bona, Sun and Zhang \cite{bsz-finite},  Molinet \cite{molinet-2},
and Molinet and Vento \cite{molinet-3}, we conjecture that the IBVP
(\ref{PP}) is locally well-posed in the space $ H^s(0,L)$ for $-1<s
\leq -\frac34 $, but ill-posed in the space $H^s (0,L)$ for any
$s<-1$.
\end{itemize}

In the literature, there is another class of IBVP of the KdV
equation posed on the finite domain $(0,L)$ as given below which has
been well studied  in the past  few years \cite{zh,
BSZ03FiniteDomain, Holmer06,bsz-finite}.
\begin{equation}\label{DPP}
        \begin{cases}
        u_t+u_x+u_{xxx}+uu_x=0,\qquad u(x,0)=\phi(x)  & x\in (0,L),  \ t\in
        \r^+ ,
        \\
        u(0,t)=h_1(t),\ u(L,x)=h_2(t),\ u_{x}(L,t)=h_3(t).
        \end{cases}
    \end{equation}
It is  interesting and constructive to compare the study of the IBVP
(\ref{DPP}) with that of the IBVP
   (\ref{PP}).

\medskip
While the study of the IBVP (\ref{DPP}) goes back as early as
\cite{Bubnov79, Bubnov80},  the nonhomogeneous IBVP (\ref{DPP}) was
first shown by Bona, Sun and Zhang \cite{BSZ03FiniteDomain}  to be
locally well-posed in the space $H^s (0,L)$ for any $s\geq 0$:

\medskip
\emph{ Let $s\geq 0$ , $r>0$ and $T>0$ be given.}
 \emph{There
exists $T^*\in (0, T]$ such that for \\ any $s-$compatible
\footnote{see \cite{BSZ03FiniteDomain}  for  the exact definition of
$s-$compatibbility.}
\[ \phi \in H^s (0,L),  \quad \vec{h}= (h_1, h_2, h_3) \in
H^{\frac{s+1}{3}}(0,T) \times H^{\frac{s+1}{3}}(0,T)\times
H^{\frac{s}{3}} (0,T) \] satisfying
\[ \| \phi \| _{H^s (0,L)} + \| \vec{h}\|_{H^{\frac{s+1}{3}}(0,T) \times H^{\frac{s+1}{3}}(0,T)\times
H^{\frac{s}{3}} (0,T)} \leq r,\] the IBVP (\ref{DPP}) admits a
unique solution
\[ u\in C([0,T^*]; H^s (0,L))\cap L^2 (0,T^*; H^{s+1} (0,L)).\]
Moreover, the corresponding solution map is Lipschitz continuous in
the corresponding spaces.}

\medskip
Later  Holmer \cite{Holmer06}  showed that the IBVP (\ref{DPP}) is
locally well-posed in the space $H^s (0,L)$ for any $-\frac34 <s<
\frac12$:

\medskip
 \emph{Let $s\in (-\frac34, \frac12)$, $r>0$
and $T>0$ be given.  here exists a $T^*\in (0, T]$ such that for any
\[ \phi \in H^s (0,L),  \quad \vec{h}= (h_1, h_2, h_3) \in
H^{\frac{s+1}{3}}(0,T) \times H^{\frac{s+1}{3}}(0,T)\times
H^{\frac{s}{3}} (0,T) \] satisfying
\[ \| \phi \| _{H^s (0,L)} + \| \vec{h}\|_{H^{\frac{s+1}{3}}(0,T) \times H^{\frac{s+1}{3}}(0,T)\times
H^{\frac{s}{3}} (0,T)} \leq r,\] the IBVP (\ref{DPP}) admits a
unique mild solution \footnote{A function $u\in C([0, T^*]; H^s
(0,L))$ is said to be a mild solution of the IBVP (\ref{DPP}) if
there  exist a sequence
\[ u_n \in C^1([0, T^*]; L^2 (0,L))\cap C[(0,T^*]; H^3 (0,L)), \quad
n=1,2, \cdots \] solving the equation in (\ref{DPP}) and as $n\to
\infty $,
\[ u_n \to u \quad in \ C([0, T^*]; H^s (0,L)),\]
\[ h_{1,n}:= u_n (0, \cdot) \to h_1 , \qquad h_{2,n} := u(L, \cdot)
\to h_2 \quad in H^{\frac{s+1}{3}}(0, T^*) \] and
\[ h_{3,n} :=\partial _x u_n (L, \cdot )\to h_3 \quad in
H^{\frac{s}{3}} (0, T^*) .\]}
\[ u\in C([0,T^*]; H^s (0,L)).\]
Moreover, the corresponding solution map is Lipschitz continuous in
the corresponding spaces.}

\medskip
More recently, Bona, Sun and Zhang \cite{bsz-finite} showed that the
IBVP (\ref{DPP}) is locally well-posed $H^s (0,L)$ for any $s>-1$.

\medskip
 \emph{Let $r>0$, $-1< s\leq 0$ and $T>0$ be given.}
\emph{There exists a $T^*\in (0, T]$ such that for any
\[ \phi \in H^s (0,L),  \quad \vec{h}= (h_1, h_2, h_3) \in
H^{\frac{s+1}{3}}(0,T) \times H^{\frac{s+1}{3}}(0,T)\times
H^{\frac{s}{3}} (0,T) \] satisfying
\[ \| \phi \| _{H^s (0,L)} + \| \vec{h}\|_{H^{\frac{s+1}{3}}(0,T) \times H^{\frac{s+1}{3}}(0,T)\times
H^{\frac{s}{3}} (0,T)} \leq r,\] the IBVP (\ref{DPP}) admits a
unique mild solution
\[ u\in C([0,T^*]; H^s (0,L)).\]
Moreover, the corresponding solution map is Lipschitz continuous in
the corresponding spaces.}

\medskip
Although there is only a slight difference between the boundary
conditions of IBVP (\ref{PP}) and the IBVP (\ref{DPP}), there is a
big gap between their well-posedness results. For the IBVP
(\ref{PP}), the well-posedness results presented in Theorem 1.2 and
Theorem 1.3 are local in the sense that the time interval $(0, T^*)$
in which the solution exists depends on $r$ and, in general, the
larger the $r$, the smaller the $T^*$. By contrast, the IBVP
(\ref{DPP}) is known to be globally well-posed in the space $H^s
(0,L)$ for any $s\geq 0$ in the sense one always has $(0,T^*)=(0,T)$
no matter how large the $r$ is (cf. \cite{BSZ03FiniteDomain,
fami-1}). The cause of this difference is that the $L^2-$energy of
the solution of the homogeneous IBVP (\ref{DPP}) $(\vec{h}=0$) is
decreasing: \
\[ \frac{d}{dt} \int ^L_0 u^2 (x,t) dx =-\frac32 u_x^2(0,t) \  for
\quad any \ t\geq 0 .\] But for the homogeneous IBVP (\ref{PP}), it
is not clear at all, in general, whether the $L^2-$energy of its
solution is increasing or decreasing since
\[ \frac{d}{dt}\int ^L_0 u^2(x,t)= -\frac32 u^2_x(0,t)+ 3u^3(L,t) \  for
\quad any \ t\geq 0 .\]

The approach used in the proof of their results in
\cite{BSZ03FiniteDomain, famin-2,  Holmer06} is very much different
from what used in the proof of Theorem A, but more or less along the
line used in the proof of Theorem B, in which the smoothing property
of the associated linear system play an important role.  In this
paper, we will use the same approach as that developed in
\cite{BSZ03FiniteDomain, BSZ06}  to prove our Theorem 1.2 and
Theorem 1.3. The key ingredients of the approach are listed below.

\medskip
\noindent (1). An explicit solution formula  will be derived for the
following  nonhomogeneous boundary value problem of the linear
equation:
\begin{equation}\label{bkdv}
\begin{cases}
v_t +v_x  +v_{xxx}  =0, \ x\in (0,L), \ t\geq 0, \\ v(x,0)=
0, \\ v(0,t)= h_1 (t), \quad v_x(L,t) =  h_2(t), \\
\qquad  v_{xx} (L,t)= h_3 (t),
\end{cases}
\end{equation}
which not only enables us to  establish the well-posedness of the
IBVP (\ref{PP}) with the optimal regularity conditions imposed on
the boundary data, but also plays  an important roles in obtaining
the well-posedness of the IBVP (\ref{PP}) in the space $H^s (0,L)$
with $-\frac34< s<0$.

\medskip
\noindent (2). The smoothing property of the associated linear
problem
\begin{equation}\label{bkdv-1}
\begin{cases}
v_t +v_x  +v_{xxx}  =f, \ x\in (0,L), \ t\geq 0, \\ v(x,0)= \phi
(x), \\ v(0,t)=h_1(t), \quad v_x(L,t) = h_2 (t),  \quad  v_{xx}
(L,t)=h_3 (t).
\end{cases}
\end{equation}
\emph{For given  $0\leq s\leq 3$ and $T>0$, there exists a constant
$C>0$ such that the solution $v$ of (\ref{bkdv-1}) satisfies
\[ \|v\|_{Z_{s,T}}\leq C \left ( \| (\phi, \vec{h})\|_{D_{s,T}}
+ \| f\|_{W^{\frac{s}{3},1}(0,T; H^s(0,L))} \right ) \]
 for any $(\phi,
\vec{h})\in D_{s,T} $ and $f\in W^{\frac{s}{3},1}(0,T; H^s(0,L))$.}
This property is an extension of the smoothing property obtained by
Colin and Ghidalia to the nonhomogeneous problem.

\medskip
 \noindent (3). Following Bona, Sun and Shuming \cite{BSZ06}, the IBVP (\ref{PP}) will
be converted to an integral equation posed on the whole line
$\mathbb{R}$ which make it possible to conduct Bourgain spaces
analysis  to obtain the well-posedness of the IBVP (\ref{PP}) in
$H^s (0,L)$ for $-\frac34<s<0$.

\medskip

 This paper is organized as follow. In  Section 2, we will study various  linear
 problems associated to the IBVP (\ref{PP}).  The  Section 3 is devoted to the  well-posedness  of
 the nonlinear
 IBVP (\ref{PP}). The
 paper is ended with some concluding remarks given in Section 4.
 Some open questions will also be listed in Section 4 for further
 investigations.

\section{Linear Problems}
\setcounter{equation}{0}
\subsection{The boundary integral operators}
Consider the nonhomogeneous boundary-value problem
   \begin{equation}\label{2-1}
        \begin{cases}
            v_t+v_x +v_{xxx}=0, \quad v(x,0)=0, \quad x \in (0,L), \quad t\geq 0.\\
            v(0,t)=h_1(t), \ v_x(L,t)=h_2(t), \ v_{xx}(L,t)=h_3(t).
        \end{cases}
    \end{equation}
We  derive  an explicit solution formula of the IBVP (\ref{2-1}).
(Without loss of generality, we assume that $L=1$ in this
subsection).  Applying the Laplace transform with respect to $t$,
(\ref{2-1}) is converted to
    \begin{equation}\label{2-2}
        \begin{cases}
            s\h{v}+ \h{v}_x+\h{v}_{xxx}=0,\\
            \h{v}(0,s)=\h{h}_1(s),\  \h{v}_x(1,s)=\h{h}_2(s),\
            \h{v}_{xx}(1,\xi)=\h{h}_3(s),
        \end{cases}
    \end{equation}
    where  $$\h{v}(x,\xi)=\int_0^{+\infty} e^{-s t}v(x,t)dt$$
    and
    \[  \hat{h}_j (s) =  \int^ {\infty}_0 e^{-st} h_j(t) dt, \quad j=1,2,3.\]
The solution of (\ref{2-2}) can be written in the form
\[ \h{v}(x,s)=\sum_{j=1}^3 c_j(s) e^{\lam_j(s)x}\]
where $\lambda _j (s), j=1,2,3$ are solutions of the characteristic
equation
\[ s +  \lam  + \lam^3=0\]
and $c_j (s), j=1,2,3$, solves the linear system
\begin{equation}\label{ls}
\underbrace{\begin{pmatrix}
1 &  1& 1\\
\lam_1 e^{\lam_1} & \lam_2 e^{\lam_2} & \lam_3 e^{\lam_3}\\
\lam_1^2 e^{\lam_1} & \lam_2^2 e^{\lam_2} & \lam_3^2 e^{\lam_3}\\
\end{pmatrix}}_{A}
\begin{pmatrix}
c_1\\ c_2\\c_3
\end{pmatrix}
= \underbrace{\begin{pmatrix} \h{h}_1\\ \h{h}_2 \\ \h{h}_3
\end{pmatrix}}_{\hat{\vec{h}}.}
\end{equation}
By Cramer's rule,
$$c_j= \frac{\Delta_j(s)}{\Delta(s)}, \ j=1,2,3$$
 with $\Delta$ the determinant of $A$ and $\Delta_j$ the determinant of the
 matrix
 $A$ with the column $j$ replaced by $\widehat{\vec{h}}$.
 Taking the inverse Laplace transform of $\widehat{v}$ and following the same arguments as
 that in \cite{BSZ03FiniteDomain} yield  the representation
 \[ v(x,t)=\sum ^3_{m=1} v_m (x,t)\]
 with
 \[ v_m (x,t)=\sum ^3_{j=1}v_{j,m}(x,t)\]
 and
 \[ v_{j,m}(x,t)= v_{j,m}^+ (x,t)+v_{j,m}^-(x,t)\]
 where
 \[ v_{j,m}^+(x,t)=\frac{1}{2\pi i}  \int_{0}^{+i\infty}  e^{s t}
 \frac{\Delta_{j,m}(s)}{\Delta(s)}  \h{h}_m(s) e^{\lam_j(s)x}ds\]
 and
 \[ v_{j,m}^-(x,t)=\frac{1}{2\pi i}  \int_{-i\infty}^{0}  e^{s t}
 \frac{\Delta_{j,m}(s)}{\Delta(s)}  \h{h}_m(s) e^{\lam_j(s)x}ds\]
 for $j, m =1,2,3$.   Here $\Delta_{j,m}(s)$ is obtained from $\Delta _j(s)$ by letting $\hat{h}_m (s) = 1$ and
$\hat{h}_k(s) = 0$ for $k \ne m, $ $k,m = 1, 2, 3$.
   Making the substitution $s=i(\rho ^3-\rho)$ with $1< \rho < \infty $ in the  the characteristic equation
   $$  s +  \lam   + \lam^3=0,$$
   the three roots are given in terms of $\rho$ by

     \begin{equation}\label{sign}
       \lam_1^+(\rho)=i\rho, \quad  \lam_2^+(\rho)= \frac{\sqrt{3\rho ^2-4}-i\rho}{2} \quad
       \lam_3^+(\rho)= \frac{-\sqrt{3\rho ^2-4}-i\rho}{2}
      .
      \end{equation}
Thus $v_{j,m}^+ (x,t)$ has the form

\[ v_{j,m}^+ (x,t)= \ds\frac{1}{2\pi}   \int^{\infty}_1  e^{i(\rho^3-\rho) t}
    \frac{\Delta_{j,m}^+(\rho)}{\Delta ^+(\rho)} \h{h}_m^+(\rho) e^{\lam_j^+(\rho)x}
    (3\rho^2-1)d\rho\]
    and \[ v_{j,m}^-(x,t)= \overline{v_{j,m}^+ (x,t)}\]
    where $\h{h}_m^+(\rho)=\h{h}_m(i(\rho ^3-\rho))$, $\Delta
    ^+(\rho)$ and $\Delta ^+_{j,m} (\rho)$ are obtained from $\Delta
    (s)$ and $\Delta _{j,m} (s)$ by replacing $s$ with $i(\rho
    ^3-\rho )$ and  $\lambda _j ^+ (\rho)=\lambda _j (i(\rho
    ^3-\rho))$.

For given $m, j=1,2,3$,  let $W_{j,m} $ be  an  operator on $H^s_0
(\mathbb{R}^+)$ defined as follows:  for any $h\in H^s _0
(\mathbb{R}^+)$,
\begin{equation}
 [W_{j,m}h](x,t) \equiv [U_{j,m} h](x,t) +\overline{[U_{j,m}h](x,t)}\label{c-1}
\end{equation}
with
\begin{eqnarray} [U_{j,m} h](x,t)& &\equiv \frac{1}{2\pi } \int ^{+\infty }_{1} e^{i(\rho ^3
-  \rho )t} e^{-\lambda ^+_j (\rho ) (1-x)}   (3\rho ^2 -1 )
[Q_{j,m}^+h](\rho )d\rho \nonumber \\
& & =  \frac{1}{2\pi } \int ^{+\infty }_{1} e^{i(\rho ^3 -  \rho )t}
e^{-\lambda ^+_j (\rho ) x'}   (3\rho ^2 -1 ) [Q_{j,m}^+h](\rho
)d\rho,  \quad (x'=1-x), \label{c-2}
\end{eqnarray}
for $j=1,2, \ m=1,2,3 $ and
\begin{equation}
 [U_{3,m} h](x,t) \equiv  \frac{1}{2\pi } \int ^{+\infty }_{1} e^{i(\rho ^3
-  \rho )t} e^{\lambda ^+_3 (\rho ) x}   (3\rho ^2 -1 )
[Q_{3,m}^+h](\rho )d\rho \label{c-3}
\end{equation}
for $m=1,2,3$. Here \[ [Q_{3,m} ^+h] (\rho ):=\frac{\Delta ^+_{3,m}
(\rho ) } {\Delta ^+ (\rho )}   \hat{h}   ^+ (\rho ), \qquad
[Q_{j,m}^+h] (\rho )=\frac{\Delta ^+_{j,m} (\rho ) } {\Delta ^+
(\rho )} e^{\lambda ^+_j (\rho )} \hat{h} ^+ (\rho )\] for $j=1,2 $
and $m=1,2,3$, $\hat{h}^+ (\rho) = \hat{h} (i(\rho ^3-\rho ))$. Then
the solution of the IBVP (\ref{2-1}) has the following
representation.

\begin{lem}
Given  $h_1, h_2$ and $h_3$, defining $\vec{h}=(h_1,h_2,h_3)$.  The
solution $v$ of the IBVP (\ref{2-1}) can be written in the form
\begin{equation}\label{bdrS}
v(x,t)=[W_{bdr}\vec{h}](x,t):= \sum_{j,m=1}^3 [W_{j,m}h_m](x,t).
\end{equation}
\end{lem}

\subsection{Linear estimates}
Consideration is first given to the IBVP of the linear equation:
\begin{equation}\label{LSP}
        \begin{cases}
            v_t+v_x   +v_{xxx}=f, & x\in (0,L)\\
            v(x,0)=\phi(x),\\
            v(0,t)=0, \ v_x(L,t)=0, \ v_{xx}(L,t)=0.
        \end{cases}
    \end{equation}
By the standard semigroup theory \cite{Pazy}, for any $\phi \in L^2
(0,L)$ and $f\in L^1 (0,T; L^2 (0,L))$, it admits a unique solution
$v\in C([0,T]; L^2 (0,L))$, which  can be written in the form
$$v(t)=W_0(t)\phi+ \int ^t_0 W_0 (t-\tau )f(\tau ) d\tau $$
where $W_0$ is  the $C_0$-Semigroup in the space $L^2(0,L)$
generated by the linear operator
$$Af=-f'''- f'$$
with the domain
$$\D(A)=\{ f\in H^3(0,L): f(0)=f'(L)=f''(L)=0  \}.$$
\begin{prop}\label{2.1} Let $T>0$ be given. There exists a constant
$C$ such that for any $\phi \in L^2(0,L)$ and $f\in L^1(0,T; L^2
(0,L))$, the corresponding solution $v$ of the IBVP (\ref{LSP})
belongs to the space  $ Z_{0,T}$ and
\begin{equation}\label{2.1-1} \| v\|_{Z_{0,T}} \leq C \left ( \| \phi
\| +\|f\|_{L^1 (0, T; L^2 (0,L))} \right ) .\end{equation}
\end{prop}
\begin{proof}  First multiplying the both sides of the equation in
(\ref{LSP}) by $2v$ and integrating over $(0,L)$ with respect to $x$
yields that
\[ \frac{d}{dt} \int ^L_0 v^2(x,t) +v^2(L,t)  + v_x^2(0,t)= 2\int ^L_0 f(x,t) v(x,t) dx.\]
Then,  multiplying the both sides of the equation in (\ref{LSP}) by
$2xv$ and integrating over $(0,L)$ with respect to $x$  yields that
\[ \frac{d}{dt} \int ^L_0 xv^2(x,t) +Lv^2(L,t)  +3\int ^L_0 v_x^2 dx
= \int ^L_0 v^2dx + \int ^1_0 f(x,t) v(x,t) dx.\]
The estimate
(\ref{2.1-1}) follows easily.
\end{proof}

Next we consider the nonhomogeneous boundary-value problem
   \begin{equation}\label{SP}
        \begin{cases}
            v_t+v_x +v_{xxx}=0, \quad x \in (0,L)\\
            v(x,0)=0,\\
            v(0,t)=h_1(t), \ v_x(L,t)=h_2(t), \ v_{xx}(L,t)=h_3(t)
        \end{cases}
    \end{equation}
We have the following estimate for the solution of the IBVP
(\ref{SP})

\begin{prop}\label{LbdrProblem}
    For given $T>0$, there exists a constant $C$ such that
    for any $\vec{h}\in \mathbb{H}^s (0,T)$, the corresponding solution $v$
    of the (\ref{SP}) belongs to the space $Z_{0,T}$ and
    \[ \| v\|_{Z_{0,T}} \leq C \| \vec{h}\|_{\mathbb{H}^s(0,T) }. \]
\end{prop}

\noindent {\bf Proof:}    As
    \[ \lam_1^+(\rho)=i\rho, \quad  \lam_2^+(\rho)= \frac{\sqrt{3\rho ^2-4}-i\rho}{2} \quad
       \lam_3^+(\rho)= \frac{-\sqrt{3\rho ^2-4}-i\rho}{2}, \]
 the asymptotic behaviors  of the ratios $
\frac{\Delta_{j,m}^+(\rho)}{\Delta^+(\rho)}$  for $\rho\to + \infty$
are listed below.

\begin{center}
\renewcommand{\arraystretch}{2}
\begin{tabular}{||>{$}c<{$}| >{$}c<{$}|>{$}c<{$}||}\hline
\frac{\Delta_{1,1}^+(\rho)}{\Delta ^+(\rho)}  \sim
e^{-\frac{\sqrt{3}}{2}\rho} & \frac{\Delta_{2,1}^+(\rho)}{\Delta
^+(\rho)}
 \sim e^{-\sqrt{3}\rho} & \frac{\Delta_{3,1}^+(\rho)}{\Delta ^+(\rho)}  \sim 1 \\
\hline \frac{\Delta_{1,2}^+(\rho)}{\Delta ^+(\rho)} \sim \rho^{-1} &
\frac{\Delta_{2,2}^+(\rho)}{\Delta ^+(\rho)} \sim
\rho^{-1}e^{-\frac{\sqrt{3}}{2}\rho} & \frac{\Delta_{3,2}^+(\rho)}{\Delta ^+(\rho)}  \sim  \rho^{-1} \\
 \hline
\frac{\Delta_{1,3}^+(\rho)}{\Delta ^+(\rho)}  \sim  \rho^{-2}
&\frac{\Delta_{2,3}^+(\rho)}{\Delta ^+(\rho)}  \sim
 \rho^{-2}e^{-\frac{\sqrt{3}}{2}\rho}  & \frac{\Delta_{3,3}^+(\rho)}{\Delta ^+(\rho)}  \sim  \rho^{-2}
   \\ \hline
\end{tabular}
\end{center}

For $m=1,2,3$ and $j=1,2$, set
$$\h{h^*}_{3,m}^+(\rho):=[Q_{3,m}^+h_m](\rho) = \frac{\Delta_{3,m}^+(\rho)}{\Delta^+(\rho)}\h{h}_m^+(\rho)
$$
 and
 \[   \h{h^*}_{j,m}^+(\rho):=[Q_{j,m}^+h_m](\rho) = \frac{\Delta_{j,m}^+(\rho)}{\Delta
^+(\rho)}e^{\lam_j^+(\rho)}\h{h}^+_m(\rho)\] and view $h^{*}_{j,m}$
as the inverse Fourier transform of $\h{h^*}_{j,m}^+$. It is
straightforward to see   that for any $s\in \mathbb{R}$,
\begin{equation}\label{estimates} \left \{
\begin{array}{lll} h_1\in H_0^{(s+1)/3}(\mathbb{R}^+) &\Rightarrow h_{j,1}^* \in
H^{\frac{s+1}{3}} (\mathbb{R}), & \ j=1,2,3,\\ \\
  h_2\in H_0^{s/3}(\mathbb{R}^+) &\Rightarrow h^*_{j,2}\in
H^{\frac{s+1}{3}} (\mathbb{R}),& \ j=1,2,3,\\ \\
 h_3\in H_0^{(s-1)/3}(\mathbb{R}^+)&\Rightarrow h_{j,3}^* \in H^{\frac{s+1}{3}}
 (\mathbb{R}), & \ j=1,2,3.
 \end{array} \right .
 \end{equation}


\renewcommand{\arraystretch}{1.8}




The proof is completed by using the same argument as that used in
the proofs of Proposition 2.7, Proposition 2.8 and Proposition 2.9
in \cite{BSZ03FiniteDomain}. $\Box$

\medskip
Combining Proposition 2.2 and Proposition 2.3 leads to the following
estimates for solutions of the IBVP
 \begin{equation}\label{DSP}
        \begin{cases}
            v_t+v_x +v_{xxx}=f, \quad x \in (0,L)\\
            v(x,0)=\phi (x),\\
            v(0,t)=h_1(t), \ v_x(L,t)=h_2(t), \ v_{xx}(L,t)=h_3(t)
        \end{cases}
    \end{equation}
    \begin{prop} Let $T>0$ and $s\in [0, 3]$ with $s\ne \frac{j}{2},
    j=1,3,5 $ be given. There exists a constant $C>0$ such that for any given $s-$compatible $(\phi,
    \vec{h})\in D_{s,T}$ and $f\in W^{\frac{s}{3},1}(0,T; L^2 (0,L))$, the IBVP (\ref{DSP}) admits a unique
    solution $v\in Z_{s,T}\cap H^{\frac{s}{3}} (0,T; H^1 (0,L))$ satisfying
    \[ \| v\|_{Z_{s,T}\cap H^{\frac{s}{3}} (0,T; H^1 (0,L))}\leq
    C\left (\| (\phi , \vec{h})\|_{D_{s,T}} +\|f\|_{W^{\frac{s}{3},1}(0,T; L^2 (0,L))}\right ).\]
    \end{prop}
    \noindent
    {\bf Proof:}  We only prove it holds for $s=0$ and $s=3$. The other case of $s$ follows by standard interpolation.
    Note that this proposition holds for $s=0$ because of
    Proposition 2.2 and Proposition 2.3. To see it holds for $s=3$,
    let $w=v_t$. Then $w$ solves
    \[ \begin{cases}
            w_t+w_x +w_{xxx}=f_t, \quad x \in (0,L)\\
            w(x,0)=\phi ^* (x),\\
            w(0,t)=h_1'(t), \ w_x(L,t)=h_2'(t), \ w_{xx}(L,t)=h_3'(t)
        \end{cases} \]
with \[\phi ^ *(x) =f_t (x,0)- \phi '''(x)-\phi '(x).\] Thus
\[ \|v_t \|_{Z_{0,T}} =\| w\|_{Z_{0,T}} \leq C\left (\| f\|_{W^{1,1}(0,T; L^2 (0,L))}
+\| (\phi , \vec{h})\|_{X_{3,T}}\right ) .\]  Since
\[ v_{xxx}=f-v_t -v_{xxx} -v_x, \]
we obtain further that
\[ \| v\|_{Z_{3,T}} \leq C\left (\| f\|_{W^{1,1}(0,T; L^2 (0,L))}
+\| (\phi , \vec{h})\|_{D_{3,T}}\right ) .\] The proof is complete.
\hfill$\Box$

Proposition 2.4  will be sufficient for us to obtain the local
well-posedness of the IBVP (\ref{PP}) int the space $H^s(0,L)$ for
$s\geq  0$. However, to obtain its well-posedness in the space $H^s
(0,L)$ with $s < 0$, we need to extend the problem posed on the
finite domain $(0,L)\times (0,T)$ to an equivalent problem posed on
the whole plane $\mathbb{R}\times \mathbb{R}$ in order to use
Bourgain space analysis.

First recall the solution of the following linear KdV equation,
\begin{equation}\label{RP}
        \begin{cases}
            v_t+v_x   +v_{xxx}=0, \qquad x \in \r,  \ t\in \r^+\\
            v(x,0)=\psi,\\
        \end{cases}
\end{equation}
has the explicit form
\begin{equation}\label{RSol}
v(x,t)= W_{\r} (t)]\psi(x)= c \int_{\r} e^{i(\xi^3-\xi)t }e^{ix \xi
} \hat{\psi} (\xi)d\xi
\end{equation}
Here $\hat{\psi }$ denotes the Fourier transform of $\psi $.

Taking advantage of this simplicity as it is done in
\cite{BSZ03FiniteDomain}, we rewrite $W_0(t)$ in term of $W_{\r}(t)$
and $W_{bdry}(t)$ as follows. For any $\phi \in H^s (0,L)$, let
$\phi ^* \in H^s (\r)$ be its standard extension from $(0,L)$ to
$\r$.  Let $v=v(x,t)$ is the solution of
\begin{equation*}
        \begin{cases}
            v_t+v_x   +v_{xxx}=0, \quad x\in \r, \ t\ge0\\
            v(x,0)=\phi^*,\\
        \end{cases}
\end{equation*}
and set  $g_1(t)=v(0,t), \  g_2(t)=v_x(L,t)$ and $g_3(t)=v_{xx}
(L,t)$, $\v{g}=(g_1,g_2,g_3)$ and
$$v_{\v{g}}=v_{\v{g}}(x,t)=[W_{bdr}(t)\v{g}](x),$$
which is the corresponding solution of the nonhomogeneous
boundary-value problem \ref{SP} with boundary data $h_j(t)=g_j(t)$
for $j=1,2,3$ and $t\ge 0.$ Then  $v(x,t) - v_{\v{g}}$ solves the
IBVP (\ref{LSP}). This leads us thus
a particular representation of $W_0(t)$ in terms of  $W_{bdr}(t)$ and $W_{\r}(t)$.\\

Let $B: H^s(0,L) \to H^s(\r)$ be a standard extension operator from
$H^s(0,L)$ to $H^s(\r)$.
\begin{lem}\label{extension}
Given $s\in \r$  and $\phi\in H^s(0,L)$, let $\phi^*=B\phi$. Then
\begin{equation}\label{decom}
    W_0(t)\phi= W_{\r}(t)\phi^*-W_{bdr}(t)\v{g}
\end{equation}
for any  $t>0$ and $x\in (0,L)$, where $\v{g}$ is obtained from the
trace of $W_{\r}(t)\phi^*$ at $x=0,L.$
\end{lem}

The solution of the non-homogeneous initial boundary-value problem
\begin{equation}\label{FSP}
        \begin{cases}
            v_t+v_x   +v_{xxx}=f(x,t), \qquad x \in (0,L),  \ t\geq 0\\
            v(x,0)=0, \\ v(0,t)=0, \qquad v_x(L,t)=0,\qquad
            v_{xx}(L,t)=0
        \end{cases}
\end{equation}
 can also be expressed in terms of $W_{\r} (t)$ and
$W_{bdr} (t)$.

\begin{lem}\label{Forcedext}
If $f^*(.,t)=Bf(.,t)$, with $B$ as was defined before the extension
of $f$ from $[0,L]\times \r^+ \to \r \times \r^+$, then the solution
$u$ of the extended problem (\ref{FSP}) is
$$v(x,t)=\int ^t_0 W_0 (t-\tau ) f(\tau ) d\tau =\int_0^t W_{\r}(t-\tau)f^*(.,\tau) d\tau - W_{bdr}(t)\v{v}$$
for any $x\in (0,L)$ and $t\ge 0$ where $\v{v}\equiv
\v{v}(t)=(v_1(t),v_2(t),v_3(t))$ is the appropriate boundary traces
of
$$q(x,t)=\int_0^t W_{\r} (t-\tau) f^*(\tau) d\tau$$
at $x=0,L$ i.e.  $$v_1(t)=q(0,t), \ v_2(t)=q_x(1,t),\
v_3(t)=q_{xx}(L,t)$$
\end{lem}

Lemma  \ref{extension} and Lemma \ref{Forcedext} are validated  when
$x\in (0,L)$ and $t\ge0$ since some of the operators that we have
constructed are defined only in this interval,
 moreover the only operator that is defined in the whole line is $W_{\mathbb{R}}(t)$ for any  values of $x$ and $t$.
  In the equation  (\ref{decom}), the left hand side is defined for
  all $x\in \r$ but the right hand side is defined just in $(0,L)$.  Since we want to use the Bourgain
  Spaces, we need to extend  the operator of the right hand side.

Recall that
\[ W_{bdr}(t)\vec{h}= \sum _{j, m=1}^3 W_{j,m } h_j \]
and each $ W_{j, m} h_j $ is  either of the form (see Lemma 2.1)

\begin{equation}\label{extWbdr}
[U_{bdr}^1(t)]h(x)= \frac{1}{2\pi}  \mathfrak{Re}  \int_{1}^{\infty}
e^{i t(\mu^3-\mu)} e^{\frac{-\sqrt{3\mu^2
-4}-i\mu}{2}x}(3\mu^2-1)\h{h}(\mu)d\mu
\end{equation}
or of the form
\begin{equation}\label{extWbdr-1}
[U_{bdr}^2(t)]h(x)= \frac{1}{2\pi}  \mathfrak{Re}  \int_{1}^{\infty}
e^{i t(\mu^3-\mu)} e^{i\mu x}(3\mu^2-1)\h{h}(\mu)d\mu
\end{equation}
 where  $\h{h}(\mu)=h(i(\mu^3-\mu))$. Therefore by the extension method introduced in \cite{BSZ06},  the
operator $W_{bdr} (t)$  can be extended as   ${\cal W}_{bdr} (t)$
with
\[ [{\cal W}_{bdr} (t)\vec{h}] (x,t)\]
 defined for any $t, \ x  \in \r $ and \[ [{\cal W}_{bdr}
(t)\vec{h}] (x,t)= [W_{bdr}\vec{h}] (x,t) \ for \ any \ (x,t)\in
(0,L)\times (0,T).\]  Moreover, the following estimates hold.

\begin{prop}
For given $\alpha > \frac12$ and  $(b,s)$ such that $s\le 0$ and
$b<1/2$ satisfying $$ 0\le b<1/2-s/3,$$ there exists a constant $C$
such that for any $T>0$ and any $\h{h}\in \H^s(0,T)$,
\[ {\cal W}_{bdr}\h{h} \in C([0,T]; H^s (0,L))\cap X^{\alpha,T}_{s,b} \]
and
$$\|{\cal W}_{bdr}\h{h}\| _{C([0,T]; H^s (0,L))\cap X^{\alpha,T}_{s,b}  }\le C\| \v{h}\|_{\H^s(0,T)}.$$
\end{prop}







The following lemmas are important in establishing the
well-posedness of of the IBVP (\ref{PP}) in $H^s (0,L)$ with $s<0$
whose proofs can be found in \cite{KPV93b,ColKen02, Holmer06,
BSZ06}.

\begin{lem}\label{Molinet}
Let  $-\i < s < \i, \ 0 < b\le 1, \ \frac12 <\alpha \leq 1$,  and
$\psi \in C^{\infty}_0 (\r)$ be given. There exists a constant $C$
depending only on $s$,  $\alpha $, $b$  and $\psi $ such that

    \begin{equation}
         \| \psi(t) W_{\r}(t) \phi \|_{{\cal X}_{s,b}^{\alpha}}\le C\| \phi\|_{H^s(\r)}
    \end{equation}
and
      \begin{equation}  \left\| \psi(t) \int_0^t W_{\r} (t-t') f(t')dt' \right\|_{{\cal X}_{s,b}^{\alpha}}\le C_{\delta} \|f
        \|_{Y_{s,1-b}^{1-\alpha}}
    \end{equation}
\end{lem}
Next we present the  spatial trace estimates of $W_{\r}(t)\phi$ and
$ \int_0^t W_{\r}(t-t') f(\cdot ,t')dt'$ whose proofs can be found
in \cite{ColKen02, Holmer06}
\begin{lem}\label{xxx}
Let $s \in [-1,2]$ be given. There exists a constant $C$ depending only on $s$ such that

\begin{equation}\label{xx13}
\sup_{x\in\r}\| W_{\r}(t)\phi \|_{H_t^{\frac{s+1}{3}}(\r)}\le \|\phi
\|_{H^s}(\r),
\end{equation}
\begin{equation}\label{xxx13}
\sup_{x\in\r}\| \partial_x W_{\r} (t)\phi
\|_{H_t^{\frac{s}{3}}(\r)}\le \|\phi \|_{H^s}(\r)
\end{equation}
and
\begin{equation}\label{xxx13}
\sup_{x\in\r}\| \partial_{xx} W_{\r} (t)\phi
\|_{H_t^{\frac{s-1}{3}}(\r)}\le \|\phi \|_{H^s}(\r)
\end{equation}
\end{lem}

\begin{lem}\label{xxxx}
Let $0\le b <1/2, \ -1 \le s\le 2, \ \psi\in C_0^{\i}(\r)$ and
$$ w(x,t)= \int_0^t W_{\r}(t-t') f(\cdot ,t')dt'$$
there exists $C$ depending only on $b,s$ and $\psi$ such that
\[
 \sup_{x\in\r} \|\psi(\cdot )w(x,.) \|_{H_t^{\frac{s+1}{3}}(\r)}\le C \|
 f\|_{Y_{s,b}},
\]
\[
 \sup_{x\in\r} \|\psi(\cdot )w_x(x,.) \|_{H_t^{\frac{s}{3}}(\r)}\le C \| f\|_{Y_{s,b}}
\]
and
\[
 \sup_{x\in\r} \|\psi(\cdot )w_{xx}(x,\cdot ) \|_{H_t^{\frac{s-1}{3}}(\r)}\le C \| f\|_{Y_{s,b}}
\]
\end{lem}
The following bilinear estimate is crucial in establishing the
well-posedness of the IBVP (\ref{PP} whose proof can be found in
\cite{KPV93b,ColKen02, Holmer06}).

\begin{lem}  Given $s> -\frac34 $, there exist $b=b(s) < \frac12  $,  $\alpha = \alpha (s) >\frac12  $  and $C, \, \mu >0$ such that
\begin{equation}
 \norm{\partial _x (uv)}_{Y_{s,b }^{\alpha }} \leq C T^{\mu }  \norm{u}_{X_{s,b}^{\alpha }}\norm{v}_{X_{s,b}^{\alpha }}\label{MM-1}
\end{equation}
for any $u, \, v\in X_{s,b}^{\alpha }$ with compact support in $[-T,
T]$.
\end{lem}
\section{Nonlinear Problem}
\setcounter{equation}{0} In this section, we consider the
well-posedness of the following nonlinear problem  in the space $H^s
(0,L)$.
\begin{equation}\label{3.1}
        \begin{cases}
            v_t+v_x+vv_x+v_{xxx}=0, & x\in (0,L), \ t>0 \\
            v(x,0)=\phi(x),\\
            v(0,t)=h_1(t), \ v_x(L,t) =h_2(t), \ v_{xx}(L,t)  =h_3(t) & t\ge
            0.
        \end{cases}
    \end{equation}
First we consider its well-posedness in the space $H^s(0,L)$ for
$s\geq 0$. Recall that for given $s\geq 0$ and $T>0$,
\[ D_{s,T}:= H^s (0,L)\times H^{\frac{s+1}{3}}(0,T)\times
H^{\frac{s}{3}} (0,T)\times H^{\frac{s-1}{3}}(0,T)\] and
\[
Z_{s,T}:=\C([0,T];H^s (0,L))\cap L^2 (0,T; H^{s+1}(0,L))  .\] In
addition,  let
\[{\cal Z}_{s,T} := Z_{s,T}\cap H^{\frac{s}{3}}(0,T; H^1(0,L)).\]
The proof of the following lemma can be found in
\cite{BSZ03FiniteDomain, kz}.
\begin{lem}\label{lemmma for F(u) estimate}
\begin{itemize}
\item[(i)]
For $s\geq0$ there exists a $C\geq0$ such that for any  $T>0$ and
$u, \ v\in Z_{s,T}$,

\begin{equation}\label{Ys,t estimate}
\int\limits_0^T\|uv_x\|_{H^s(0,L)}\;d\tau\leq C
(T^{\frac{1}{2}}+T^{\frac13})\|u\|_{Z_{s,T}}\|v\|_{Z_{s,T}}
\end{equation}

\item[(ii)] For $0\leq s\leq 3$ there exists a $C\geq0$ such that for any  $T>0$ and
$u, \ v\in {\cal Z}_{s,T}$,

\begin{equation}\label{forcing}
\| uv_x\|_{W^{\frac{s}{3},1} (0,T; L^2 (0,1))} \leq C
(T^{\frac{1}{2}}+T^{\frac13})\|u\|_{{\cal Z}_{s,T}}\|v\|_{{\cal
Z}_{s,T}}
\end{equation}
\end{itemize}
\end{lem}
\begin{thm} Let $T>0$, $r>0$ and $s\geq 0$ be given with
$s\ne \frac{2j+1}{2}$ for $j=0,1,2, \cdots$.  There exists a $T^*\in
(0,T]$ such that for any $s-$compatible $(\phi , \vec{h}) \in X_{s,
T}$, the IBVP (\ref{3.1}) admits a unique solution
\[ v\in Z_{s,T^*}.\]
Moreover, the corresponding solution map is Lipschitz continuous.
\end{thm}
\noindent
 {\bf Proof:} Since the proof is similar to that  presented
in \cite{BSZ03FiniteDomain}, we only provide a sketch and refer to
\cite{BSZ03FiniteDomain} for more detail.

\medskip
1). We first consider the case of  $0\leq s\leq 3$. Let $r>0$ and
$0< \theta \leq \max \{1, T\}$ be constants to be determined. Set \[
S_{\theta , \eta }=\{ w\in {\cal Z}_{s, \theta}, \quad \| w\|_{{\cal
Z}_{s, \theta}} \leq \eta \} .\] For given $(\phi , \vec{h})\in
D_{s,T}$ with
\[ \| (\phi, \vec{h})\|_{D_{s,T}} \leq r, \]
define a nonlinear map  on $S_{\theta , \eta }$ by
\[ v=\Gamma (w) \]
being the unique solution of the IBVP
\[ \begin{cases}
            v_t+v_x +v_{xxx}=-ww_x , & x\in (0,L), \ t>0 \\
            v(x,0)=\phi(x),\\
            v(0,t)=h_1(t), \ v_x(L,t)=  h_2(t), \ v_{xx}(L,t)= h_3(t) & t\ge
            0
            \end{cases}
 \]
 for $w\in {\cal Z}_{\theta , \eta}$. Using Proposition 2.4  and
 Lemma 3.1, one can show that there exist $\eta >0$ and $\theta >0$
 depending only on $s$, $r$ and $T$ such that the map $\Gamma $ is a
 contraction on the  metric space $S_{\theta ,\eta}$ whose fixed
 point is the desired solution for the IBVP (\ref{3.1}) Thus the
 theorem hods for $0\leq s\leq 3$.

 \medskip
 2). Next we consider the following IBVP of the linear KdV equation
 with variable coefficients.

\begin{equation}\label{3.2} \begin{cases}
            v_t+v_x +v_{xxx}+(a v)_x =f, & x\in (0,L), \ t>0 \\
            v(x,0)=\phi(x),\\
            v(0,t)=h_1(t), \ v_x(L,t)-v(L,t)=h_2(t), \ v_{xx}(L,t)-v(L,t)=h_3(t) & t\ge
            0
            \end{cases}
 \end{equation}
 As in the step 1, using the contraction mapping principle, one can
 show the following  proposition holds.

 \begin{prop}  Let $T>0$ and $0\leq s\leq 3$ be given and assume
that $a\in {\cal Z}_{s,T}$. Then for any $s-$compatible $(\phi,
\vec{h})\in D_{s,T}$ and $f\in W^{\frac{s}{3},1}(0,T; L^2 (0,L))$,
the IBVP (\ref{3.2}) admits a unique solution $v\in {\cal Z}_{s,T}$.
Moreover, there exists a constant $C>0$ depending only on $T$ and
$\|a\|_{{\cal Z}_{s,T}}$ such that
\[ \| v\|_{{\cal Z}_{0,T}} \leq C \left (\| (\phi, \vec{h})\|_{D_{s,T}} +
\|f \| _{W^{\frac{s}{3},1} (0,T; H^s (0,L))}\right ).\]
\end{prop}

\medskip
3). Now we prove the theorem hold for $3\leq s\leq 6$. The other of
$s>6 $ follows similarly. First of all, according to Step 2), the
IBVP (\ref{3.1}) admits a unique solution $uv\in {\cal Z}_{3,T^*}$.
We just need to prove this solution $v$ also belong to the space
${\cal Z}_{s,T^*}$.  To see that, let $z=v_t$. Then $z$ solves the
following linearized IBVP
\[ \left \{ \begin{array}{l} z_t +z_x +(a(x,t)z)_x +z_{xxx}=0,  \\
z(x,0)= \phi _1 (x),  \\  z(0,t)=h^{(1)}_1 (t), \quad z_x (L,t) =
h^{(1)}_2 (t), \quad z_{xx} (L,t) =h^{(1)}_3 (t)
\end{array}\right.\] where $a(x,t)=v(x,t)\in {\cal
Z}_{3,T^*}$ and \[ \phi _1 \in H^{s-3} (0,L), \ h^{(1)}_1  \in
H^{\frac{s-2}{3}}(0,T^*), \quad , \ h^{(1)}_2\in
H^{\frac{s-3}{3}}(0,T^*), \quad  h^{(1)}_3 \in
H^{\frac{s-4}{3}}(0,T^*).\] It thus follows from Proposition 3.3
that
\[ z=v_t \in {\cal Z}_{s-3, T^*}\]
and therefore
\[ v\in {\cal Z}_{s, T^*}\]
since
\[ v_{xxx}= -v_t -v_x -vv_x .\]

\hfill $\Box$

 Next we consider the well-posedness of the IVP (\ref{3.1}) in the space $H^s (0,L)$ with
 $s<0$. We first rewrite the IBVP (\ref{3.1})  in its integral
 form;
 \begin{equation}\label{3.3}
 v(t)=W_0 (t) \phi +W_{bdr} (t) \vec{h}  - \int ^t_0
 W_0(t-\tau ) (vv_x) (\tau )d\tau .
 \end{equation}
\begin{thm} Let $T>0 $, $r>0$ and $-\frac34<s < 0$ be given. There exists
a $T^*\in (0, T]$  and $b\in (0, \frac12)$ such that for any $(\phi
, \vec{h})\in X_{s,T}$, (\ref{3.3}) admits a unique solution
\[ v\in C([0,T^*], H^s (0,L))\cap Y_{s,b}^{T^*} .\]
Moreover, the corresponding solution map is Lipschitz continuous.
\end{thm}

The following lemmas are needed to prove Theorem 3.4. Let
\[ {\cal X}_{s,b}^{\alpha,T}:= C([0,T]; H^s (0,L))\cap X^{\alpha,T}_{s,b} .\]
\begin{lem} Let $T>0$, $s<0$, $\frac12 < \alpha \leq 1$ and $b\in (0,1) $ be given
satisfying
\[ 0<b< \frac12 -\frac{s}{3}.\]
  For any $\phi \in H^s (0,L)$, $ W_0 (t)\phi \in {\cal
X}^{\alpha, T}_{s,b} $ and
\[ \| W_0 (t) \phi \|_{{\cal X}_{s,b}^{\alpha, T}}\leq C\|\phi \|_{H^s (0,L)}
\]
where $C>0$ is independent of $\phi $.
\end{lem}
\noindent {\bf Proof:} It follows from Lemma 2.5, Lemma 2.8, Lemma
2.9, and Proposition 2.7.

\begin{lem}
Assume that  $-1\leq s< 1$, $\frac12 < \alpha \leq 1 $ and $0< b <
\frac12 $. For any $T>0 $, there is a constant $C$  such that
 for any $f\in Y^{1-\alpha, T}_{s,b} $, $$ u=\int ^t_0 W_0 (t-\tau ) f(\tau) d\tau \in {\cal X}_{s,b}^{\alpha ,T}
$$ and satisfies the inequality
\begin{equation} \norm{u}_{  {\cal X} _{s,b}^{\alpha, T } } \leq C \norm{f }_{Y^{1-\alpha, T}_{s,b} }. \label{y-7}
\end{equation}
In addition, there exists a $b^* \in (0, \frac12)$ such that if
$f\in Y^{1-\alpha, T}_{s, b^* }$, then $u$  belongs to the space
${\cal X}_{s,\frac12 }^{\alpha , T} $  and satisfies the bound
\begin{equation} \norm{u}_{  {\cal X} _{s,\frac12 }^{\alpha, T } } \leq C \norm{f }_{Y^{1-\alpha, T}_{s,b^*} }. \label{y-8}
\end{equation}

\end{lem}

\noindent {\bf Proof:}  It follows from Lemma 2.6, Lemma 2.8, Lemma
2.10, and Proposition 2.7.

\begin{lem}
Given $T>0$,  $s> -\frac34 $, there exist $b=b(s) < \frac12  $,
$\alpha = \alpha (s) >\frac12  $  and $C, \, \mu >0$ such that
\begin{equation}
 \norm{\partial _x (uv)}_{Y_{s,b }^{\alpha , T} } \leq C T^{\mu }  \norm{u}_{{\cal X}_{s,b}^{\alpha, T } }
 \norm{v}_{{\cal X}_{s,b}^{\alpha ,T } }\label{M-1}
\end{equation}
for any $u, \, v\in {\cal X}_{s,b}^{\alpha , T} $.
\end{lem}
\noindent {\bf Proof:} It follows from Lemma 2.11 directly.

\medskip

 Now we at the stage to present of the proof of Theorem 3.4

\medskip
 \noindent
{\bf Proof of Theorem 3.4:}

\medskip
 For given $(\phi,\v{h})\in D_{s,T} $ and  $s\in (-\frac34,0)$, let   $\theta\in (0,1]$ to be determined.
Define  $\Gamma:  {\cal X}_{s,\frac12}^{\alpha, \theta}   \to {\cal
X}_{s,\frac12}^{\alpha, \theta} $ such that
$$\Gamma(\omega) :=  W_0(t)\phi + W_{bdr}(t)\v{h} - \int_0^t W_0(t-\tau)\big(\omega \omega _x)(\tau) d\tau$$
By Lemmas 3.5, Lemma 3.6 and Lemma 3.7, we have

\begin{eqnarray*}
\| \Gamma(\omega)\|_{{\cal X}_{s,\frac12}^{\alpha, \theta}}
&\le& \|  W_0(t)\phi\|_{{\cal X}_{s,\frac12}^{\alpha, \theta}}  + \|W_{bdr}(t)\v{h}\|_{{\cal X}_{s,\frac12}^{\alpha, \theta}} \\
& & + \| \int_0^t W_0(t-\tau)\big(\omega \omega _x)(\tau)\|_{{\cal X}_{s,\frac12}^{\alpha, \theta}} \\
& \le&  C  \|(\phi , \vec{h})\|_{D_{s,T}} + C \theta^{\mu} \|\omega\|_{{\cal X}_{s,\frac12}^{\alpha, \theta}}^2\\
\end{eqnarray*}
Let  $r= 2 C  \|(\phi , \vec{h})\|_{D_{s,T}}$  and the ball
$$B_r: \{w\in \y_{s,b}^{\theta}): \|w\|_{{\cal X}_{s,\frac12}^{\alpha, \theta}}\le r \}$$
\begin{eqnarray*}
\| \Gamma(\omega)\|_{{\cal X}_{s,\frac12}^{\alpha, \theta}} &\le&   r/2 + C\theta^{\mu} r^2\\
&\le& r\left(1+C\theta^{\mu}r\right)\\
&\le&r/2 +r/2 =r.
\end{eqnarray*}
when we select  $T^*= \theta>0$ and $ 2C(T^*)^{\mu}r <1$.\\

Therefore, $$\Gamma(B_r) \subset B_r.$$

Similarly,  taking $v,\omega\in \y_{s,b}^{T^*}$,
\begin{eqnarray*}
\|\Gamma(v)- \Gamma(\omega) \|_{{\cal X}_{s,b}^{\alpha,T^*}}&\le&
C\theta^{\mu}\| v-\omega\|_{{\cal X}_{s,b}^{\alpha,T^*}}   \| v+\omega\|_{{\cal X}_{s,b}^{\alpha,T^*}}\\
&\le&   C\theta^{\mu}\| v-\omega\|_{{\cal X}_{s,b}^{\alpha,T^*}}\left( \| v\|_{{\cal X}_{s,b}^{\alpha,T^*}}
+\| \omega\|_{{\cal X}_{s,b}^{\alpha,T^*}}\right)\\
&\le& 2r C\theta^{\mu}\| v-\omega\|_{\y_{s,b}^{T^*}} \\
&\le& \beta \| v-\omega\|_{{\cal X}_{s,b}^{\alpha,T^*}}
\end{eqnarray*}
with $\beta= 2 C (T^*)^{\mu}r<1$  as we have defined before. Then,
by the contraction mapping theorem, the fixed point $u$ is the
unique solution of (\ref{3.3}).
\section{Concluding remarks}
\setcounter{equation}{0} The focus of our discussion has been the
well-posedness of the initial value problem
    of the KdV equation posed on the finite interval $(0,L)$:

    \begin{equation}\label{4.1}
        \begin{cases}
        u_t+u_x+u_{xxx}+uu_x=0,\qquad    x\in (0,L),  \ t >0,
        \\ u(x,0)=\phi(x), \\
        u(0,t)=h_1(t),\ u_x(L,x)=h_2(t),\ u_{xx}(L,t)=h_3(t).
        \end{cases}
    \end{equation}
    It is considered with the initial data $\phi \in H^s (0,L)$ and
    the boundary data $\vec{h}=(h_1, h_2 , h_3 )$ belongs to the
    space $D_{s,T}:= H^{\frac{s+1}{3}}(0,T)\times H^{\frac{s}{3}}
    (0,T)\times H^{\frac{s-1}{3} } (0, T)$ with $s>-\frac34$. Using
    the approaches developed in \cite{BSZ03FiniteDomain, BSZ06} we have succeeded in showing
    that the IBVP (\ref{4.1}) is locally well-posed in the space
    $H^s(0,L)$ for any $s>-\frac 34$ with $s\ne \frac{2j+1}{2}$,
    $j=0,1,2,\cdots$, which extends and improve the earlier works of Colin and
    Ghidalia \cite{c-g-1, c-g-2, ColGhi01}.

    \smallskip
    However, the well-posedness results presented in
    Theorem 1.1 and Theorem 1.3 are  \emph{conditional} in the sense
    that the uniqueness holds in a stronger Banach space than that
    of $C([0.T]; H^s (0,L))$.  In particular, in the case of $s\geq
    0$,  according to Theorem 1.1, the uniqueness holds in the space
    \[ C([0,T]; H^s (0,L))\cap L^2 (0,T; H^{s+1}(0,L))\]
    rather than in the space $C([0,T]; H^s (0,L))$.  Also when
    $-\frac34 < s< 0$, Theorem 1.3 states that the uniqueness holds
    in the space
    \[ C([0,T]; H^s (0,L))\cap X^{\alpha, T}_{s, \frac12} \]
    which is a stronger subspace of  $C([0,T]; H^s (0,L))$. A question arises
    naturally:

    \medskip
    \emph{Does the uniqueness hold in the space $C([0,T]; H^s
    (0,L))$?}

    \medskip
    If the uniqueness does hold in the space $C([0,T];H^s (0,L))$,
    then the corresponding well-posedness is called \emph{unconditional
    well-posedness. } (The interesting readers are referred to  Bona, Sun and Zhang
    \cite{BSZ04}
    for conditional and unconditional
    well-posedness of evolution equations.)

    By using the usual energy estimate method, one can show  that
    the uniqueness does hold for the IBVP (\ref{4.1}) in the space
    $C([0,T]; H^s (0,L))$ when $s>\frac32$. Thus the IBVP
    (\ref{4.1}) is unconditionally (locally) well-posed in the space
    $H^s(0,L)$ for any $s>\frac32 $ with
    \[ s\ne \frac{2j+1}{2}, \quad j=1,2,\cdots \]
    The following question remains open.

    \medskip
    \noindent
    {\bf Question 4.1:} \emph{Is the IBVP (\ref{4.1}) unconditionally
    well-posed in the space $H^s (0,L)$ for some $s< \frac32$?
}

By contrast, the IBVP
\begin{equation}\label{4.2}
        \begin{cases}
        u_t+u_x+u_{xxx}+uu_x=0,\qquad    x\in (0,L),  \ t >0,
        \\ u(x,0)=\phi(x), \\
        u(0,t)=h_1(t),\ u(L,x)=h_2(t),\ u_{x}(L,t)=h_3(t).
        \end{cases}
    \end{equation}
is known to be \emph{unconditionally well-posed} in the space $H^s
(0,L)$ for any $s>-1$. This is because that the IBVP (\ref{4.2}) is
known to  be globally well-posed in the space $H^s (0,L)$ for any
$s\geq 0$. In particular, its classical solutions exist globally.
However, the IBVP (\ref{4.1}) is only known to be locally
well-posed. Whether it is globally well-posed is still an open
question.

\medskip
\noindent {\bf Question 4.2:} \emph{Is the IBVP (\ref{4.1}) globally
well-posed in the space $H^s (0,L)$ for some $s\geq 0$?}

\medskip
This is the same  question asked  earlier by Colin and Ghidalia
\cite{ColGhi01}. They showed that  that if $\phi \in H^1 (0,L)$ and
$h_j \in C^1 (\r ^+)$, $j=1,2,3$ are small enough, then the
corresponding solution  $u$ of (\ref{4.1}) exists globally:
\[ u\in L^{\infty} (\r^+; H^1 (0,L)).\]

\medskip
Recently, Rivas, Usman and Zhang \cite{ruz} showed  that the
solutions of the IBVP (\ref{4.1}) exist globally (in time) in
    the space $H^s(0,L)$  for any $s\geq 0$ as long as its  auxiliary data
    $(\phi , \vec {h})$ is small in the space $D^s_T$.
    In addition, they  have shown  that those small amplitude
    solutions decay exponentially if their boundary value
    $\vec{h}(t)$ decays exponentially. In particular, those
    solutions satisfying homogenous boundary conditions decay
    exponentially in the space $H^s (0,L)$ if their initial values
    are small in  $H^s (0,L)$.

Note that a positive answer to Question 4.2 leads  to a positive
answer to Question 4.1  using the general approach developed by
Bona, Sun and Zhang \cite{BSZ04} for establishing unconditional
well-posedness of nonlinear evolution equations.

\smallskip
Recently, Bona, Sun and Zhang \cite{bsz-finite} showed that the IBVP
(\ref{4.2}) is locally  (unconditionally) well-posed in the space
$H^s (0,L)$ for any $s>-1$. One of the key steps in their approach
is to transfer the IBVP (\ref{4.2}) of the KdV equation to an
equivalent IBVP of the KdV-Burgers equation.  Precisely,
 let
\[ u(x,t)=e^{2t-x} v (x,t).\]
Then  $u$ is a solution of the IBVP (\ref{4.2}) if and only if $v$
is a solution of the following IBVP of the KdV-Burgers equation
posed on the finite interval $(0,L)$:
\begin{equation}\label{kdv-b}
\begin{cases}
v_t +4v_x -3v_{xx}+v_{xxx}+ e^{2t-x}(vv_x -v^2) =0, \ x\in (0,L), \
t\geq 0, \\ v(x,0)= \phi (x)e^x, \\ v(0,t)=e^{-2t+L}h_1 (t), \quad
V(L,t)= e^{-2t+L}h_2 (t), \\ \quad  v_x(L,t)= e^{-2t+L} h_3(t)+h_1
(t) e^{-2t+L}.
\end{cases}
\end{equation}
Consequently, one can adapt the approach of Molinet \cite{molinet}
in dealing with the pure initial value problems  of the KdV-Burgers
equation posed either on the whole line $\r$ or on a periodic domain
$\mathbb{T}$ to show that the IBVP (\ref{kdv-b}) is locally
well-posed in the space $H^s (0,L)$ for any $s>-1$. However, the
same transformation converts the IBVP¡¡(\ref{4.1}) to the following
IBVP of the KdV-Burgers equation
\begin{equation}\label{kdv-b-1}
\begin{cases}
v_t +4v_x -3v_{xx}+v_{xxx}+ e^{2t-x}(vv_x -v^2) =0, \ x\in (0,L), \
t\geq 0, \\ v(x,0)= \phi (x)e^{x}, \\ v(0,t)=e^{-2t+L}h_1 (t), \quad
v_x(L,t)-v(L,t)= e^{-2t+L} h_2(t), \\ \qquad  v_{xx}
(L,t)-v(L,t)=e^{-2t+L}(2h_2 (t) +h_3 (t)).
\end{cases}
\end{equation}
Note that the boundary conditions of (\ref{kdv-b-1}) are different
from those of (\ref{kdv-b}). That brings a challenge  to show that
the IBVP (\ref{kdv-b-1}) to be locally well-posed  in $H^s (0,L)$
for $s>-1$. The following question thus remains to be open.

\medskip
\noindent {\bf Question 4.3}: \emph{Is  the IBVP well-posed in the
space $H^s (0,L)$ for $-1< s\leq -\frac34?$}

\medskip

Finally we would like to point out that the KdV equation including,
in particular, the IBVP (\ref{4.2}) has been extensively studied in
the past twenty years from control point of view (cf
\cite{KRZ91,RZ93, RZ95, Rosier97, zh, Za, coron, cerpa-1, cerpa-2,
zuazua, pazoto,rosier-z} and the reference therein). The interested
readers are specially referred to Rosier and Zhang \cite{Za-survey}
for a recent survey on this subject. By contrast, the study of the
IBVP (\ref{4.1}) is still widely open. It will be  very interesting
to see if there are any differences between the IBVP (\ref{4.1}) and
the  the IBVP (\ref{4.2}) from control point view.

\bigskip
\textbf{Acknowledgments}. Ivonne Rivas was partially supported by
the Taft Memorial Fund at the University of Cincinnati through
Graduate Dissertation Fellowship. Bing-Yu Zhang was partially
supported  by the Taft Memorial Fund at the University of
Cincinnati. The work was partially conducted while the second author
(IR) and the third author (BZ) were participating  the trimester
program, \emph{Control of Partial Differential Equations and
Applications}, held at the Institut Henri Poincar\'e  (Paris) from
October 1, 2010 to December 18, 2010.  They thank the Institute for
its hospitality  and financial support.

\end{document}